\documentclass[aps,preprint,groupedaddress]{revtex4-2}

\usepackage{indentfirst}
\usepackage{amsmath}
\usepackage{amsfonts}
\usepackage{amssymb}
\usepackage{graphicx}
\usepackage{titlesec}
\usepackage{amsthm}
\usepackage{color}
\usepackage{cancel}
\usepackage[normalem]{ulem}%%%%

\newtheorem{theorem}{Theorem}
\newtheorem{proposition}{Proposition}
\titleformat{\section}
  {\normalfont\fontsize{14}{15}\bfseries}{\thesection}{1em}{}
\newtheorem{innercustomthm}{Proposition}
\newenvironment{reprop}[1]
  {\renewcommand\theinnercustomthm{#1}\innercustomthm}
  {\endinnercustomthm}
  
 \definecolor{masoncolor}{rgb}{0.98, 0.27, 0.62}

\begin{document}

\title{Low-Dimensional Behavior of a Kuramoto Model with Inertia and Hebbian Learning}

\author{Tachin Ruangkriengsin}
\affiliation{Department of Mathematics, University of California Los Angeles, Los Angeles, California 90095, USA}

\author{Mason A. Porter}
\affiliation{Department of Mathematics, University of California Los Angeles, Los Angeles, California 90095, USA}
\affiliation{Santa Fe Institute, Santa Fe, New Mexico, 87501, USA}

%%%%%

\begin{abstract}

We study low-dimensional dynamics in a Kuramoto model with inertia and Hebbian learning. In this model, the coupling strength between oscillators depends on the phase differences between the oscillators and changes according to a Hebbian learning rule. We analyze the special case of two coupled oscillators, which yields a five-dimensional dynamical system that decouples into a two-dimensional longitudinal system and a three-dimensional transverse system. We readily write an exact solution of the longitudinal system, and we then focus our attention on the transverse system. We classify the stability of the transverse system's equilibrium points using linear stability analysis. We show that the transverse system is dissipative and that all of its trajectories are eventually confined to a bounded region. We compute Lyapunov exponents to infer the transverse system's possible limiting behaviors, and we demarcate the parameter regions of three qualitatively different behaviors. Using insights from our analysis of the low-dimensional dynamics, we study the original high-dimensional system in a situation in which we draw the intrinsic frequencies of the oscillators from Gaussian distributions with different variances.

\end{abstract}

%%%%%

\maketitle

%%%%

%%%%%%

{\bf Synchronization occurs ubiquitously in many systems \cite{scholarped-sync}, such as in electrical impulses of neurons and in the flashing of fireflies. In response to changes in external environmental conditions, many systems of coupled oscillators adapt to enhance collective behavior. One example of adaptation is plasticity in networks of neurons, where the synaptic strengths between neurons change based on their relative spike times or on other features \cite{mateos2019}. Another example of adaptation occurs in groups of fireflies~\cite{ermentrout_1991}, which have been modeled using coupled phase oscillators with inertia. In situations like these two examples, adaptation can accelerate the synchronization of nearby oscillators while simultaneously facilitating global synchronization. In our work, we consider both plasticity and inertia by studying a modified version of the ubiquitous Kuramoto model of coupled oscillators \cite{strogatz2000}. To examine the interplay between oscillator plasticity and inertia, we mathematically analyze a small system of these coupled oscillators and use the results of this analysis to gain insights into larger systems of such oscillators.}

%%%%

%%%

\section{Introduction}\label{Introduction}

The analysis of systems of coupled oscillators has been used extensively to study collective behavior in many situations \cite{scholarped-sync}, such as flashing in groups of fireflies \cite{mirollo_strogatz_1990}, synchronization of pedestrians on the Millennium Bridge \cite{strogatz_abrams_mcrobie_eckhardt_ott_2005,belyk2021}, pacemaker conduction \cite{michaels_matyas_jalife_1987}, and singing in frogs \cite{frogs2020}. In 1975, Yoshiki Kuramoto proposed a model of coupled biological oscillators as a system of a first-order differential equations in which each variable corresponds to the phase of an oscillator \cite{kuramoto_1984}. In this model, which is now known as the Kuramoto model, he assumed that the instrinsic frequency of each oscillator is chosen from a fixed probability distribution. He also assumed all-to-all coupling of these oscillators and that each oscillator is influenced by each other oscillator by an amount that is proportional to the sine of their phrase difference.

One natural way to extend the original Kuramoto model is by incorporating a second-order term to include inertia \cite{rodrigues2016}. This extension was proposed as part of an adaptive model to explain the ability of the firefly \textit{Pteroptyx malaccae} to synchronize its flashing with almost no phase difference \cite{ermentrout_1991}. Researchers have also employed Kuramoto models with inertia to study disordered arrays of Josephson junctions \cite{trees_saranathan_stroud_2005}, decentralized power grids \cite{rohden_sorge_timme_witthaut_2012}, and a variety of other phenomena. 
Following the setup of Tanaka et al. \cite{tanaka_lichtenberg_oishi_1997}, we consider a Kuramoto model with inertia of the form
\begin{equation} \label{kuramoto_general_equation}
    m_i\frac{d^2 \phi_i}{dt^2} +\frac{d \phi_i}{dt} = \omega_i + \frac{1}{N} \sum_{j=1 \,, \, j \neq i}^{N}K_{ij}\sin(\phi_{j}-\phi_{i})\,, \quad i \in \{1,2,\ldots,N\}\,,
\end{equation}
where $\phi_{i} \in [0,2\pi)$ is the phase of the $i$th oscillator, $\omega_i$ is its intrinsic frequency, $m_i$ is its mass, $N$ is the number of oscillators, and $K_{ij} = K_{ji}$ is the symmetric coupling strength between the $i$th and $j$th oscillators. 

In the original formulation of the Kuramoto model, the coupling strength $K_{ij}$ is constant. However, this assumption is too restrictive for some problems. For example, in neuroscience, neurons exhibit synaptic plasticity when their strengths change in response to activity-dependent mechanisms \cite{mateos2019}. According to Hebbian theory \cite{hebb_1964}, the synaptic strength between two neurons increases when they are active simultaneously. If we view regularly-spiking neurons as coupled oscillators, the time-dependent coupling strength $K_{ij} = K_{ij}(t)$ increases when the phase difference between oscillators $i$ and $j$ decreases. 

Adaptive Kuramoto models (without inertia) have been studied extensively in the past decade \cite{ghosh2022,berner2023adaptive,sawicki2023}. They have very rich dynamics, such as rich bifurcation structures~\cite{juttner2022}, the coexistence of multiple distinct clusters of oscillators~\cite{berner_multiclusters}, and ``heterogeneous nucleation'' (i.e., with both single-step and multi-step transitions to synchronization) \cite{fialkowski2022}. Kuramoto models with Hebbian learning have also been studied with multiplex \cite{berner2020} and polyadic \cite{kach2022} interactions between oscillators. Rigorous mean-field \cite{kuehn_meanfield} and continuum \cite{kuehn_continuum} limits have been developed to study adaptive network dynamics in a broad class of Kuramoto models (without inertia). 

Researchers have proposed a variety of functions to model Hebbian changes in the coupling between oscillators~\cite{niyogi_english_2009,timms_english_2014,seliger_young_tsimring_2002,ren_zhao_2007,ha_noh_park_2016, aoki_aoyagi, berner_multiclusters, berner_adaptive}  in Kuramoto models without inertia. We follow~\cite{niyogi_english_2009} and suppose that the coupling strengths satisfy
\begin{equation} \label{hebb}
    \frac{dK_{ij}}{dt} = \beta(\alpha \cos(\phi_{j}-\phi_{i})-K_{ij}) \,,
\end{equation}
where $\alpha > 0$ is the \textit{learning enhancement factor} and $\beta>0$ is the \textit{learning rate}. Because $\cos(\phi_{j}-\phi_{i}) = \cos(\phi_{i}-\phi_{j})$, we let $K_{ij}(0) = K_{ji}(0)$ so that $K_{ij}(t) = K_{ji}(t)$ for all $t \geq 0$ is the symmetric coupling strength between the $i$th and $j$th oscillators. Equations (\ref{kuramoto_general_equation}, \ref{hebb}) constitute a dynamical system, with
$N + \frac{(N-1)(N)}{2} = \frac{N(N+1)}{2}$ equations, of coupled Kuramoto oscillators with inertia and Hebbian learning. We first consider the case $N = 2$ and study the resulting low-dimensional system. We then use the insights from our analysis with $N = 2$ to briefly consider a higher-dimensional system (with $N = 50$).

Our paper proceeds as follows. In Section \ref{LSA_section}, we perform linear stability analysis of equations (\ref{kuramoto_general_equation}, \ref{hebb}) when there are $N = 2$ oscillators. In Section \ref{diss_contr_section}, we verify its dissipation and contraction properties. In Section \ref{low_dim_section}, we demarcate the different behaviors of this system in a two-dimensional parameter space. We briefly examine the coupled oscillator system for $N = 50$ oscillators in Section \ref{high_dim_section}, and we conclude in Section \ref{conc}.

%%%%

\section{Linear Stability Analysis of a System of $N = 2$ Coupled Oscillators}
\label{LSA_section}

We examine equations (\ref{kuramoto_general_equation}, \ref{hebb}) with $N = 2$ oscillators with identical masses. This yields
the {five-dimensional (5D) dynamical system} 
\begin{align} 
    m\frac{d^2 \phi_1}{dt^2} +\frac{d \phi_1}{dt} &= \omega_1 + \frac{1}{2} k\sin(\phi_{2}-\phi_{1})\,, \label{general_reduce1}  \\
    m\frac{d^2 \phi_2}{dt^2} +\frac{d \phi_2}{dt} &= \omega_2 + \frac{1}{2} k\sin(\phi_{1}-\phi_{2})\,, \label{general_reduce2}  \\
    \frac{dk}{dt} &= \beta(\alpha \cos(\phi_{2}-\phi_{1})-k)\,, \label{general_reduce3}
\end{align}
where $k := K_{12} = K_{21}$.

To analyze the dynamical system (\ref{general_reduce1})--(\ref{general_reduce3}), we consider the transverse coordinate  $\phi := \phi_1 - \phi_2$ and longitudinal coordinate $\psi := \phi_1 + \phi_2$. In prior studies of synchronization, it has been very insightful to analyze dynamical systems using such coordinates \cite{pikovsky1991,olmi2015}. Taking the difference and sum of (\ref{general_reduce1}, \ref{general_reduce2}) yields
\begin{align}
	m\frac{d^2 \phi}{dt^2} + \frac{d \phi}{dt} &= \omega_1 - \omega_2 - k \sin(\phi) \,, \label{transversal} \\
	m\frac{d^2 \psi}{dt^2} + \frac{d \psi}{dt} &= \omega_1 + \omega_2 \label{longitudinal}
\end{align}
In these new coordinates, the 5D dynamical system (\ref{general_reduce1})--(\ref{general_reduce3}) decouples into two independent systems: a three-dimensional (3D) dynamical system (\ref{general_reduce3}, \ref{transversal}) in the transverse direction and a two-dimensional (2D) dynamical system \eqref{longitudinal} in the longitudinal direction. 

Equation \eqref{longitudinal} is a second-order ordinary differential equation with constant coefficients. Its solution is
\begin{equation}\label{longitudinal_eq}
	\psi(t) = C_1 + (\omega_1 + \omega_2)t + C_2e^{-\frac{1}{m}t} \,,
\end{equation}
where the initial values $\psi(0)$ and $\psi'(0)$ determine the integration constants $C_1$ and $C_2$. As $t \rightarrow \infty$, we see that $\frac{d\psi}{dt} \approx \omega_1 + \omega_2$ for any initial values. Therefore, the sum of the oscillators' phases eventually increases approximately linearly as a function of time.

The transverse system (\ref{general_reduce3}, \ref{transversal}) does not have an exact solution, so we write it as a 3D dynamical system to analyze its behavior. We specify the domains of the phase difference $\phi \in [-\pi,\pi)$, the derivative $\gamma := \frac{d\phi}{dt} \in \mathbb{R}$ of the phase difference, and the intrinsic-frequency difference $\omega := \omega_1-\omega_2 \in \mathbb{R}$. By symmetry, we assume without loss of generality that $\omega \geq 0$. We reduce the number of parameters in (\ref{general_reduce3}, \ref{transversal}) by rescaling time and defining
$\tilde{k} := {k}/{\beta}$, $\tilde{\alpha}={\alpha}/{\beta}$, $\tilde{\omega}:= {\omega}/{\beta}$, $\tilde{\gamma} := {\gamma}/{\beta}$, $\tilde{m} := \beta m$, and $\tilde{t} := \beta t$. These transformations indicate that we do not need the parameter $\beta$, so we take $\beta = 1$ without loss of generality and write
\begin{align} 
    \frac{d\phi}{dt} &= \gamma \,, \label{reduce1}\\
    \frac{d\gamma}{dt} &= \frac{1}{m}(-\gamma+\omega-k\sin(\phi)) \,, \label{reduce2}\\
        \frac{dk}{dt} &= \alpha\cos(\phi)-k \,, \label{reduce3}
\end{align}
where we drop the tildes from our notation for convenience.

We obtain the equilibrium points $(\phi^*,\gamma^*,k^*)$ of the dynamical system (\ref{reduce1})--(\ref{reduce3}) by setting $\frac{d\phi}{dt}=\frac{d\gamma}{dt} = \frac{dk}{dt}=0$. This implies that $\gamma^* = -\gamma^*+\omega-k^*\sin(\phi^*) = 0$ and $\alpha\cos(\phi^*)-k^* = 0$. Simplifying yields
\begin{equation}\label{sin_identity}
    \sin(2\phi^*) = \frac{2\omega}{\alpha} \,, 
\end{equation}
so the equilibrium points exist if and only if $\alpha \geq 2\omega$. Because $\phi \in [-\pi,\pi)$, we obtain four equilibria: 
$P_1 =   (\frac{1}{2}\arcsin(\frac{2\omega}{\alpha}),0,\alpha\cos(\frac{1}{2}\arcsin(\frac{2\omega}{\alpha})))$, $P_2 =  (\frac{\pi}{2}-\frac{1}{2}\arcsin(\frac{2\omega}{\alpha}),0,\alpha\sin(\frac{1}{2}\arcsin(\frac{2\omega}{\alpha})))$, $P_3 = (-\pi + \frac{1}{2}\arcsin(\frac{2\omega}{\alpha}),0,-\alpha\cos(\frac{1}{2}\arcsin(\frac{2\omega}{\alpha})))$, and $P_4 = (-\frac{\pi}{2}-\frac{1}{2}\arcsin(\frac{2\omega}{\alpha}),0,-\alpha\sin(\frac{1}{2}\arcsin(\frac{2\omega}{\alpha})))$.

The Jacobian matrix of the linearization of (\ref{reduce1})--(\ref{reduce3}) at the equilibrium points is
\begin{equation}\label{Jacobian}
J(\phi^*,\gamma^*,k^*) = 
    \begin{bmatrix} 
    0 & 1 & 0\\ 
    -\frac{k^*}{m}\cos(\phi^*)& -\frac{1}{m} & -\frac{\sin(\phi^*)}{m}\\
    -\alpha\sin(\phi^*) & 0 & -1
    \end{bmatrix}
     = 
    \begin{bmatrix} 
    0 & 1 & 0\\ 
    -\frac{\alpha \cos^2(\phi^*)}{m}& -\frac{1}{m} & -\frac{\sin(\phi^*)}{m}\\
    -\alpha\sin(\phi^*) & 0 & -1
    \end{bmatrix}\,.
\end{equation}
The eigenvalues $\lambda$ of $J(\phi^*,\gamma^*,k^*)$ satisfy
\begin{equation}\label{eigen_equation}
    -\lambda(\lambda+1)\left(\lambda+\frac{1}{m}\right)-\frac{\alpha\cos^2(\phi^*)}{m}(\lambda+1) + \frac{\alpha\sin^2(\phi^*)}{m} = 0\,.
\end{equation}
Using (\ref{sin_identity}), we simplify equation (\ref{eigen_equation}) for each equilibrium point.
The equilibrium points $P_1$ and $P_3$ both give 
    \begin{equation}\label{eig_P1P3}
       -2\lambda(\lambda+1)(m\lambda+1)-(\alpha+\sqrt{\alpha^2-4\omega^2})(\lambda+1) + (\alpha-\sqrt{\alpha^2-4\omega^2}) = 0 \,.
    \end{equation}
The equilibrium points $P_2$ and $P_4$ both give 
    \begin{equation}\label{eig_P2P4}
       -2\lambda(\lambda+1)(m\lambda+1)-(\alpha-\sqrt{\alpha^2-4\omega^2})(\lambda+1) + (\alpha+\sqrt{\alpha^2-4\omega^2}) = 0 \,.
    \end{equation}

When $\alpha = 2\omega$, we obtain $P_1 = P_2$ and $P_3 = P_4$. These mergers of equilibrium points are saddle--node bifurcations that arise from the same characteristic equation: 
\begin{equation}
    \lambda(\lambda+1)(m\lambda+1)+\frac{\alpha \lambda}{2}= \lambda\left(m\lambda^2+(m+1)\lambda+1+\frac{\alpha}{2}\right) = 0 \,,
\end{equation}
which gives $\lambda = 0$ and $\lambda  = \frac{-(m+1) \pm \sqrt{(m-1)^2-2m\alpha} }{2m}$.

When $\alpha > 2\omega$, we have the following proposition. 

\medskip

\begin{proposition}\label{prop_eig}
 Consider the dynamical system (\ref{reduce1})--(\ref{reduce3}) with $\alpha > 2\omega$. Let $u := \alpha + \sqrt{\alpha^2-4\omega^2}$ and $v := \alpha - \sqrt{\alpha^2-4\omega^2}$.
 The following statements hold:
\begin{enumerate}
    \item Let $\Gamma_1$ be the region in the $(u,v)$ plane with $0 \leq v \leq u \leq \frac{2(m^2-m+1)}{3m}$ that is bounded by the curves
    \begin{align*}
        v &= u-\frac{(m+1+\sqrt{4(m+1)^2-6m(u+2)})(\sqrt{4(m+1)^2-6m(u+2)}-2(m+1))^2}{54m^2} \,, \\
        v &= u-\frac{(m+1-\sqrt{4(m+1)^2-6m(u+2)})(\sqrt{4(m+1)^2-6m(u+2)}+2(m+1))^2}{54m^2} \,.
    \end{align*}
    If $(u,v) \notin \Gamma_1$, then the Jacobian matrix at the equilibria $P_1$ and $P_3$ has a negative real eigenvalue and two complex-conjugate eigenvalues with negative real part. Otherwise, the Jacobian matrix at the equilibria $P_1$ and $P_3$ has three negative real eigenvalues.
    \item Let $\Gamma_2$ be the region in the $(u,v)$ plane with $0 \leq v \leq u$ and $v \leq \frac{2(m^2-m+1)}{3m}$ that is bounded by the curves
    \begin{align*}
        u = v-\frac{(m+1+\sqrt{4(m+1)^2-6m(v+2)})(\sqrt{4(m+1)^2-6m(v+2)}-2(m+1))^2}{54m^2} \,, \\
        u = v-\frac{(m+1-\sqrt{4(m+1)^2-6m(v+2)})(\sqrt{4(m+1)^2-6m(v+2)}+2(m+1))^2}{54m^2} \,.
    \end{align*}
    If $(u,v) \notin \Gamma_2$, the Jacobian matrix at the equilibria $P_2$ and $P_4$ has a positive real eigenvalue and two complex-conjugate eigenvalues with negative real part. Otherwise, the Jacobian matrix at the equilibria $P_2$ and $P_4$ has one positive real eigenvalue and two negative real eigenvalues.
\end{enumerate}
\end{proposition}

We prove Proposition \ref{prop_eig} in Appendix \ref{app}. By performing numerical computations on a uniform grid in the $(u,v)$ plane, we construct a region $\Gamma_1$ in the $(u,v)$ plane for which the Jacobian matrix at the equilibria $P_1$ and $P_3$ has three real eigenvalues. We let $(u,v) \in [0, \frac{2(m^2-m+1}{3m}] \times [0, \frac{2(m^2-m+1)}{3m}]$, divide this rectangle into a uniform grid with $1000 \times 1000$ points, and calculate the eigenvalues of the Jacobian matrix at the equilibria $P_1$ and $P_3$ at each grid point. In Figure \ref{real_eig}, we plot the region $\Gamma_1$ and its boundary. The boundary of the region $\Gamma_1$ matches well with the boundary of the region that we obtain with numerical simulations. The same is true for the region $\Gamma_2$.

\begin{figure}[htp]
\centering
\includegraphics[width= \textwidth]{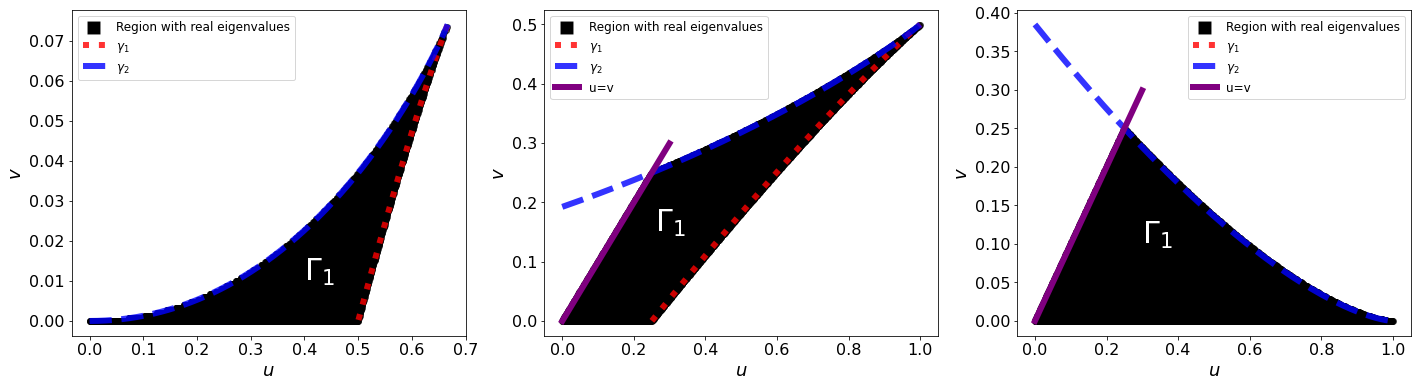}
\caption{The region $\Gamma_1$ in the $(u,v)$ plane with (left) $m = 1$, (center) $m = 2$, and (right) $m = 0.5$. In this region, the Jacobian matrix at the equilibria $P_1$ and $P_3$ has negative real eigenvalues. We use $\gamma_1$ to denote the curve $u = v-\frac{(m+1+\sqrt{4(m+1)^2-6m(v+2)})(\sqrt{4(m+1)^2-6m(v+2)}-2(m+1))^2}{54m^2}$ and $\gamma_2$ to denote the curve $u = v-\frac{(m+1-\sqrt{4(m+1)^2-6m(v+2)})(\sqrt{4(m+1)^2-6m(v+2)}+2(m+1))^2}{54m^2}$.}
\label{real_eig}
\end{figure}

%%%%%

\section{Dissipation and Contraction of the Dynamical System (\ref{reduce1})--(\ref{reduce3})}\label{diss_contr_section}

A dynamical system $\frac{d\vec{x}}{dt} = f(\vec{x})$ is dissipative if the volume of any fixed region of phase space contracts as a function of time. When the divergence $\nabla \cdot f < 0$ is constant, the volume contracts exponentially fast with rate $\nabla \cdot f$. Calculating the divergence $\nabla \cdot f$ that is associated with equations (\ref{reduce1})--(\ref{reduce3}) gives
\begin{equation*}
	\frac{\partial}{\partial \phi}(\gamma) + \frac{1}{m}\frac{\partial}{\partial \gamma}(-\gamma + \omega - k\sin{\phi}) + \frac{\partial}{\partial k}(\alpha \cos{\phi}-k) = -\frac{1}{m}-1 < 0\,.
\end{equation*}	 
Therefore, our low-dimensional {transverse} system is dissipative, with a volume contraction rate of $-\frac{1}{m} - 1$. We also show that all trajectories of (\ref{reduce1})--(\ref{reduce3}) are eventually confined to a bounded region of phase space. 

\begin{theorem}\label{thm_traj_bdd}
Suppose that $(\phi(t),\gamma(t),k(t))_{t \geq 0}$ is a trajectory of the dynamical system (\ref{reduce1})--(\ref{reduce3}) and that $(\phi(0),\gamma(0),k(0)) = (\phi_0,\gamma_0,k_0)$. 
It is then the case that for all $\epsilon > 0$, there exists some time $T_{\epsilon} \geq 0$ such that $|k(t)| \leq \alpha+\epsilon$ and $|\gamma| \leq \omega + \alpha + \epsilon$ for all $t \geq T_{\epsilon}$. 
\end{theorem}

\begin{proof}
We start by multiplying both sides of (\ref{reduce3}) with the integrating factor $e^t$ to obtain 
\begin{align*}
    \frac{d}{dt}\left(k(t)e^{t}\right) = \alpha \cos{\phi(t)}e^{t} &\implies k(t)e^{t}-k(0) = \int_{0}^{t}\alpha \cos{\phi(\tilde{t})}e^{\tilde{t}}\,d\tilde{t} \\
    &\implies |k(t)e^{t}-k_0| \leq \int_{0}^{t}\alpha |\cos{\phi(\tilde{t})}|e^{\tilde{t}}\,d\tilde{t} \leq \int_{0}^{t}\alpha e^{\tilde{t}}\,d\tilde{t} = \alpha\left(e^{t}-1\right)\,.
\end{align*}
Consequently, $|k(t)| \leq \alpha + \frac{|k_0|-\alpha}{e^{t}} \leq \alpha + \epsilon$ for all  $t \geq T_1 = \ln\left(\left|\frac{|k_0|-\alpha}{\epsilon}\right|+1\right)$. 
Similarly, we multiply both sides of (\ref{reduce2}) with the integrating factor $e^{{t}/{m}}$ to obtain
\begin{align*}
    \frac{d}{dt}(\gamma e^{{t}/{m}}) &= \frac{\omega e^{{t}/{m}}}{m} - \frac{k(t)e^{{t}/{m}}\sin{\phi(t)}}{m} \quad \\
        &\hspace{-1cm} \Longrightarrow \gamma(t)e^{{t}/{m}}-\gamma_0 = \int_{0}^{t}\left[\frac{\omega e^{{\tilde{t}}/{m}}}{m}-\frac{k(\tilde{t})e^{{\tilde{t}}/{m}}\sin{\phi(\tilde{t})}}{m}\right]d\tilde{t}\,.
\end{align*}
Let $T_2 \in \mathbb{R}$ such that $|k(t)| \leq \alpha + \frac{\epsilon}{2}$ for all $t \geq T_2$. For all $t \geq T_2$, we then have
\begin{align*}
    \left|\gamma(t)e^{\frac{t}{m}}-\gamma_0\right| &\leq \int_{0}^{t}\left|\frac{\omega e^{{\tilde{t}}/{m}}}{m}-\frac{k(\tilde{t})e^{{\tilde{t}}/{m}}\sin{\phi(\tilde{t})}}{m}\right|\,d\tilde{t} \\ 
    &\leq \int_{0}^{t}\frac{\omega e^{{\tilde{t}}/{m}}}{m}\,d\tilde{t} + \int_{0}^{t}\frac{|k(\tilde{t})e^{{\tilde{t}}/{m}}|}{m} \,d\tilde{t} \\
    &\leq \omega(e^{{t}/{m}}-1)+ \int_{0}^{T_2}\frac{|k(\tilde{t})e^{{\tilde{t}}/{m}}|}{m} \, d\tilde{t} + \int_{T_2}^{t}\frac{(\alpha+\frac{\epsilon}{2})e^{{\tilde{t}}/{m}}}{m}\,d\tilde{t} \\
    &= \omega(e^{{t}/{m}}-1)+ \int_{0}^{T_2}\frac{|k(\tilde{t})e^{{\tilde{t}}/{m}}|}{m} \, d\tilde{t} + \left(\alpha+\frac{\epsilon}{2}\right)(e^{{t}/{m}}-e^{{T_2}/{m}}) \\
    &= \left(\omega+\alpha+\frac{\epsilon}{2}\right)e^{{t}/{m}}-\omega-\left(\alpha+\frac{\epsilon}{2}\right)e^{{T_2}/{m}}+ \int_{0}^{T_2}\frac{|k(\tilde{t})e^{{\tilde{t}}/{m}}|}{m} \, d\tilde{t}\,.
\end{align*}
Consequently, 
\begin{equation*}
	|\gamma(t)| \leq \left(w+\alpha+\frac{\epsilon}{2}\right) + \frac{\left|-\omega-\left(\alpha+\frac{\epsilon}{2}\right)e^{{T_2}/{m}}+ \int_{0}^{T_2}\frac{|k(\tilde{t})e^{{\tilde{t}}/{m}}|}{m}\,  d\tilde{t}+|\gamma_0|\right|}{e^{{t}/{m}}} \,. 
\end{equation*}	
The numerator of the second term is constant, so
\begin{equation*}
    \lim_{t \to \infty}\frac{\left|-\omega-\left(\alpha+\frac{\epsilon}{2}\right)e^{{T_2}/{m}}+ \int_{0}^{T_2}\frac{|k(\tilde{t})e^{{\tilde{t}}/{m}}|}{m} \,d\tilde{t}+|\gamma_0|\right|}{e^{{t}/{m}}} = 0\,.
\end{equation*}
Therefore, there exists a constant $T_3 \in \mathbb{R}$ such that  
\begin{equation*}
	\frac{\left|-\omega-\left(\alpha+\frac{\epsilon}{2}\right)e^{{T_2}/{m}}+ \int_{0}^{T_2}\frac{|k(\tilde{t})e^{{\tilde{t}}/{m}}|}{m} \, d\tilde{t}+|\gamma_0|\right|}{e^{{t}/{m}}} \leq \frac{\epsilon}{2}
\end{equation*}	
for all $t \geq T_3$.  We conclude that $k(t) \leq \alpha + \epsilon$ and $\gamma(t) \leq \omega + \alpha + \epsilon$ for all $t \geq T_\epsilon = \max{\{T_1,T_2,T_3\}}$, as desired. 
\end{proof}

We have just shown that the dynamical system (\ref{reduce1})--(\ref{reduce3}) is dissipative and that all of its trajectories are eventually confined to a bounded region. To obtain further insight into the possible limiting behaviors of (\ref{reduce1})--(\ref{reduce3}), we compute its Lyapunov exponents near the origin with a numerical approach~\cite{sandri1996numerical} that is based on the algorithms in~\cite{benettin1980lyapunov}. We show these Lyapunov exponents in Figures \ref{lyap} and \ref{lyap_vary}.

\begin{figure}[htp]
\centering
\includegraphics[width= 0.6\textwidth]{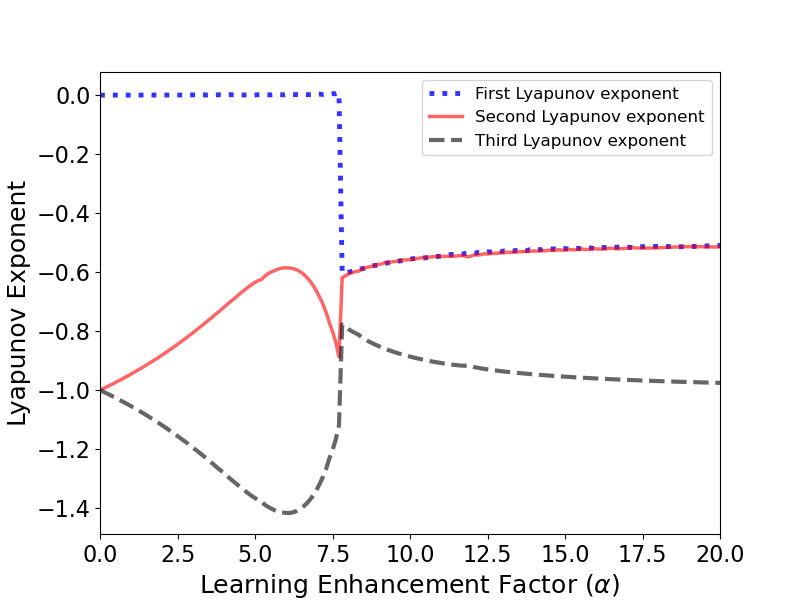}
\caption{Lyapunov exponents for trajectories near the origin $(0,0,0)$ for the dynamical system (\ref{reduce1})--(\ref{reduce3}) with $m = 1$ and $\omega = 3$ and different values of the learning enhancement factor $\alpha$.} 
\label{lyap}
\end{figure}

\begin{figure}[htp]
\centering
\includegraphics[width= \textwidth]{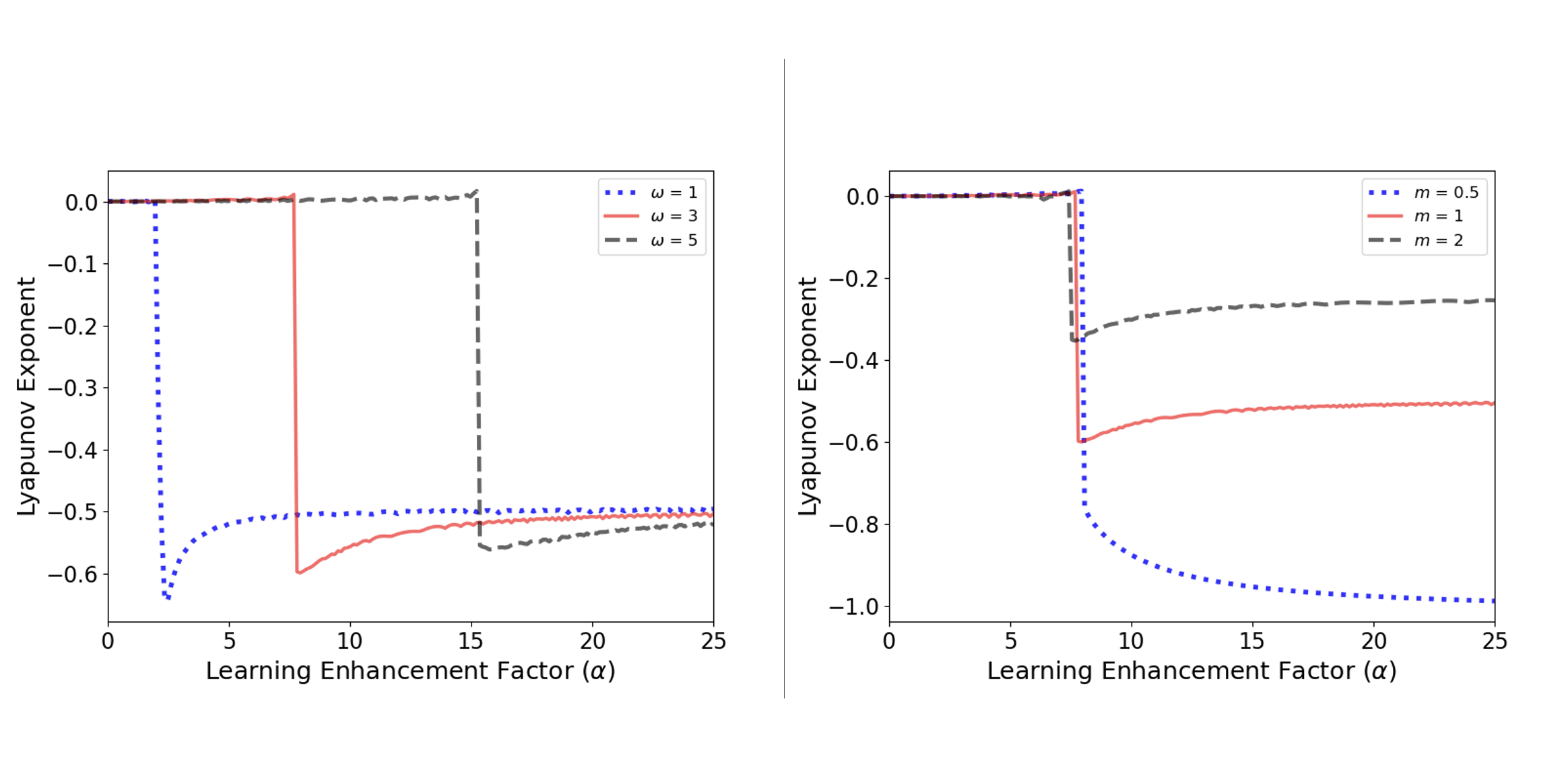}
\caption{Largest Lyapunov exponent for trajectories near the origin $(0,0,0)$ for the dynamical system (\ref{reduce1})--(\ref{reduce3}). (Left) We fix the inertia to $m = 1$ and consider intrinsic-frequency difference of $\omega = 0.1$, $\omega = 0.3$, and $\omega = 1$. (Right) We fix the intrinsic-frequency difference to $\omega = 1$ and consider inertias of $m = 0.5$, $m = 1$, and $m = 2$.}
\label{lyap_vary}
\end{figure}

In our computations, we obtain either three negative Lyapunov exponents or one $0$ Lyapunov exponent and two negative Lyapunov exponents. A chaotic attractor requires one positive Lyapunov exponent, and a quasiperiodic orbit on a 2-torus requires two zero Lyapunov exponents~\cite{klein1991}, so the attractors of the dynamical system (\ref{reduce1})--(\ref{reduce3}) must consist of equilibrium points, periodic orbits, or unions of periodic orbits and equilibrium points. When we increase the learning enhancement factor $\alpha$ for fixed inertia $m = 1$ and intrinsic-frequency difference $\omega = 3$, we move from a region with two negative Lyapunov exponents to a region with three negative Lyapunov exponents. To examine how the transition from two negative Lyapunov exponents to three negative Lyapunov exponents depends on other parameters, we compute the largest Lyapunov exponents for several values of $m$ and $\omega$ (see Figure \ref{lyap_vary}). For larger values of $\omega$, the transition occurs at larger values of $\alpha$, whereas the inertia $m$ mostly affects only the magnitude of the largest Lyapunov exponent.

To further examine the contraction of the dynamical system (\ref{reduce1})--(\ref{reduce3}), we define the energy function
\begin{equation}\label{functional}
    E(\phi,\gamma, k) := \frac{\alpha m \gamma^2}{2} -\alpha \omega \phi - \alpha k \cos{\phi} + \frac{k^2}{2}\,. 
\end{equation}
Consider the critical points of $E$. At these points, 
\begin{align*}
    \frac{\partial E}{\partial \phi} = 0 \quad &\Longrightarrow \quad -\alpha \omega +\alpha k \sin{\phi} = 0\,, \\
    \frac{\partial E}{\partial \gamma} = 0 \quad &\Longrightarrow \quad \alpha m \gamma = 0\,, \\
    \frac{\partial E}{\partial k} = 0 \quad &\Longrightarrow \quad -\alpha \cos{\phi}+k = 0 \,.
\end{align*}
These three equations are the same equations for the equilibrium points of the dynamical system (\ref{reduce1})--(\ref{reduce3}). 
Therefore, the equilibrium points are the critical points of $E$. 
In the interior of the domain $(\phi, \gamma, k) \in [-\pi, \pi) \times \mathbb{R} \times \mathbb{R}$, we calculate
\begin{align*}
    \frac{d}{dt}E(\phi,\gamma, k) &= \alpha m \gamma \frac{d\gamma}{dt}-\alpha \omega \frac{d\phi}{dt}-\alpha\left(-k\sin{\phi}\frac{d\phi}{dt} + \cos{\phi}\frac{dk}{dt}\right)+k\frac{dk}{dt} \\
    &= \alpha m \gamma \frac{d\gamma}{dt}+(\alpha k\sin{\phi} -\alpha \omega)\frac{d\phi}{dt}+(k-\alpha\cos{\phi})\frac{dk}{dt} \\
    &= \alpha \gamma(-\gamma + \omega -k\sin{\phi}) + (\alpha k \sin{\phi} - \alpha \omega)\gamma + -(k-\alpha \cos{\phi})^2 \\
    &= -(\alpha \gamma^2 + (k-\alpha \cos{\phi})^2) \leq 0 \,.
\end{align*}
We thus see that the energy of a trajectory of the dynamical system (\ref{reduce1})--(\ref{reduce3}) never increases with time and that the time derivative of the energy is independent of $m$. The energy $E$ is a Lyapunov functional in the interior of the region $\{(\phi, \gamma, k) | (\phi, \gamma, k) \in [-\pi, \pi) \times \mathbb{R} \times \mathbb{R}\}$. Because of the term $\alpha \omega \phi$, the energy $E$ is not $2\pi$-periodic in $\phi$. Therefore, its derivative on the boundary of the domain in $\phi$ is not well-defined, so it is difficult to analyze the global behavior of the system using only the energy function \eqref{functional}.

%%%%%

\section{Demarcation of Different Qualitative Dynamics in the $(\alpha,\omega)$ Plane}\label{low_dim_section}

In Section \ref{diss_contr_section}, we computed Lyapunov exponents of the dynamical system (\ref{reduce1})--(\ref{reduce3}) and observed that they depend on the
values of $\alpha$ and $\omega$. In this section, we examine how the qualitative dynamics of (\ref{reduce1})--(\ref{reduce3}) depend on the parameters $\alpha$ and $\omega$.

We simulate 50 trajectories of the dynamical system (\ref{reduce1})--(\ref{reduce3}) with initial conditions that we choose uniformly at random in $[-\pi,\pi)^3$. We set the mass of each oscillator to $m = 1$. We choose the domain of $\phi \in [-\pi,\pi)$ so that the oscillator phases satisfy $2\pi$-periodicity, and we choose the domains of $\gamma$ and $k$ for simplicity. By varying the learning enhancement factor $\alpha$ and the intrinsic-frequency difference $\omega$, we obtain three regions $\Omega_1$, $\Omega_2$, and $\Omega_3$ in the $(\alpha,\omega)$ plane in which the trajectories exhibit qualitatively different dynamics.

In the region $\Omega_1$, the dynamical system (\ref{reduce1})--(\ref{reduce3}) does not have any equilibrium points. Our simulations suggest that all trajectories converge to a periodic solution. See Figure \ref{sim_omega1} for an illustration. From equation (\ref{sin_identity}), we infer that this region occurs when $0 < \alpha < 2\omega$. However, we have not proven rigorously that all trajectories converges to a single periodic solution, and we also have not proven whether or not this periodic solution is a limit cycle. To gain insight into this periodic solution, we use an approximation. In \cite{menck_heitzig_kurths_schellnhuber_2014}, Menck et al. approximated solutions near a limit cycle of a second-order power-grid model by assuming that oscillator phases rotate at a constant frequency. Inspired by this idea, we suppose that there exists $\zeta > 0$ (which we will determine later) such that $\phi(t) \approx \zeta t + \phi(0)$ and $\frac{d\phi}{dt} \approx \zeta$; we aim to parametrize $\gamma(t)$ and $k(t)$ by $\phi(t)$. Observe that equation (\ref{reduce3}) includes the term $\cos \phi$. Therefore, we posit an approximation of $k(t)$ of the form $k(t) \approx a\cos{\phi(t)}+b\sin{\phi(t)}$ for some constants $a$ and $b$. Using the approximation $\frac{d\phi}{dt} \approx \zeta$ yields $\frac{dk}{dt} \approx (-a\sin{\phi}+b\cos{\phi})\zeta$. Inserting the approximations of $k(t)$ and $\frac{dk}{dt}$ into equation (\ref{reduce3}) gives
\begin{equation*}    
    (-a\sin{\phi}+b\cos{\phi})\zeta + (a\cos{\phi}+b\sin{\phi}) \approx \alpha \cos{\phi} \,.
\end{equation*}    
We now equate the coefficients of $\cos{\phi}$ and $\sin{\phi}$ to obtain $-a\zeta + b = 0$ and $b\zeta + a = \alpha$, and we then solve these two equations to obtain $a = \frac{\alpha}{\zeta^2+1}$ and $b = \frac{\alpha \zeta}{\zeta^2+1}$. 

From equation (\ref{reduce2}) with $m=1$, we have $\frac{d\gamma}{dt} + \gamma = \omega-k\sin{\phi}$. Similarly to our calculation above, we observe that equation (\ref{reduce2}) includes the term $k \sin{\phi}$. Recall from equation (\ref{reduce1}) that $\frac{d \phi}{dt} = \gamma$, so we posit an approximation of $\gamma(t)$ of the form $\gamma(t) \approx \zeta + c\cos{2\phi}+d\sin{2\phi}$ for some constants $c$ and $d$. Inserting the approximations of $k(t)$ and $\gamma(t)$ into equation (\ref{reduce2}) with $m=1$ yields
    \begin{align*}
	    \zeta + (-2c\zeta +d)\sin{2\phi}+(2d\zeta+c)\cos{2\phi} &\approx \omega -\left(\frac{\alpha}{\zeta^2+1}\cos{\phi}+ \frac{\alpha \zeta}{\zeta^2+1}\sin{\phi}\right)\sin{\phi} \\
    &\approx \omega - \frac{\alpha}{2(\zeta^2+1)}\sin{2\phi} - \frac{\alpha \zeta}{\zeta^2+1}\left(\frac{1-\cos{2\phi}}{2}\right) \\
    &\approx \omega - \frac{\alpha \zeta}{2(\zeta^2+1)}-\frac{\alpha}{2(\zeta^2+1)}\sin{2\phi} + \frac{\alpha \zeta}{2(\zeta^2+1)}\cos{2\phi} \,.
    \end{align*}
Equating the coefficients of $\cos{2\phi}$ and $\sin{2\phi}$ gives $\zeta = \omega - \frac{\alpha \zeta}{2(\zeta^2+1)}$, $-2c\zeta + d = -\frac{\alpha}{2(\zeta^2+1)}$, and $2d\zeta + c = \frac{\alpha \zeta}{2(\zeta^2+1)}$. Therefore, $c = \frac{3\zeta \alpha}{2(\zeta^2+1)(4\zeta^2+1)}$ and $d = \frac{(2\zeta^2-1)\alpha}{2(\zeta^2+1)(4\zeta^2+1)}$, where $\zeta$ is a real root of $2\zeta^3-2\omega \zeta^2 + (\alpha+2)\zeta -2\omega = 0$.

In summary, our approximation of the periodic solution satisfies
\begin{align*}
    \phi(t) &\approx \zeta t + \phi(0) \,, \\
    \gamma(t) &\approx \zeta + \frac{3\zeta \alpha}{2(\zeta^2+1)(4\zeta^2+1)}\cos{2\phi(t)}+\frac{(2\zeta^2-1)\alpha}{2(\zeta^2+1)(4\zeta^2+1)}\sin{2\phi(t)} \,, \\
    k(t) &\approx \frac{\alpha}{\zeta^2+1}\cos{\phi(t)}+\frac{\alpha\zeta}{\zeta^2+1}\sin{\phi(t)}\,,
\end{align*}
where $\zeta$ is a real root of $2x^3-2\omega x^2 + (\alpha+2)x - 2\omega = 0$. We have checked numerically that this polynomial equation always has a single real root, so $\zeta$ is unique. As we can see in the right panel of Figure \ref{sim_omega1}, our approximate periodic solution is reasonably accurate. We have checked numerically that the phase difference $\phi$ increases approximately linearly with time. Therefore, our assumptions approximately hold in practice.

\begin{figure}[htp]
\centering
\includegraphics[width=\textwidth]{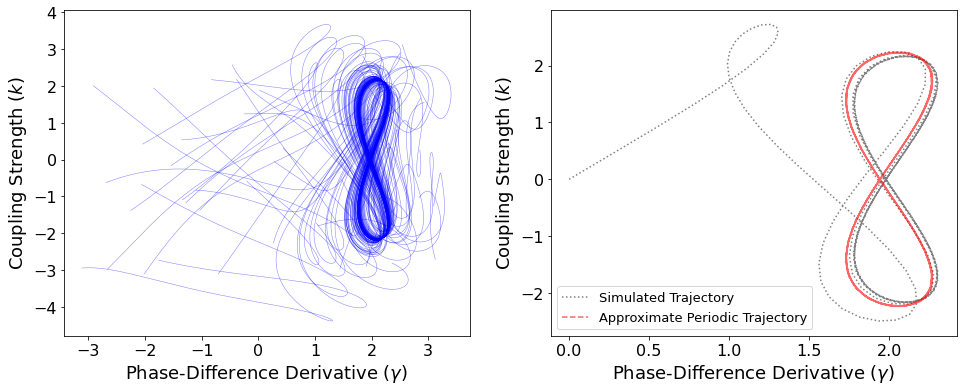}
\caption{Projection onto the $(\gamma,k)$ plane of simulated trajectories in the region $\Omega_1$ for the dynamical system (\ref{reduce1})--(\ref{reduce3}) of $N = 2$ coupled oscillators with parameters $m=1$, $\omega=3$, and $\alpha = 5$. (Left) We simulate 50 trajectories with initial values that we choose uniformly at random in $[-\pi,\pi)^3$. (Right) We approximate the periodic trajectory that we observe in the region $\Omega_1$.
}
\label{sim_omega1}
\end{figure}

In the region $\Omega_2$, the dynamical system (\ref{reduce1})--(\ref{reduce3}) has four equilibrium points for 
\linebreak
$(\phi, \gamma, k) \in [-\pi, \pi) \times \mathbb{R} \times  \mathbb{R}$. They are $P_1 =   (\frac{1}{2}\arcsin(\frac{2\omega}{\alpha}),0,\alpha\cos(\frac{1}{2}\arcsin(\frac{2\omega}{\alpha})))$, \linebreak
$P_2 =  (\frac{\pi}{2}-\frac{1}{2}\arcsin(\frac{2\omega}{\alpha}),0,\alpha\sin(\frac{1}{2}\arcsin(\frac{2\omega}{\alpha})))$, $P_3 = (-\pi + \frac{1}{2}\arcsin(\frac{2\omega}{\alpha}),0,-\alpha\cos(\frac{1}{2}\arcsin(\frac{2\omega}{\alpha})))$, and $P_4 = (-\frac{\pi}{2}-\frac{1}{2}\arcsin(\frac{2\omega}{\alpha}),0,-\alpha\sin(\frac{1}{2}\arcsin(\frac{2\omega}{\alpha})))$.
In this region, there exists a heteroclinic orbit that connects the equilibrium points $P_2$ and $P_4$. This situation is rather different from the periodic dynamics that we observed in the region $\Omega_1$. As one can see in our simulations in the left panel of Figure \ref{sim_omega2}, some trajectories converge to the equilibrium points but others converge to this heteroclinic orbit. By contrast, in region $\Omega_3$ (see the right panel of Figure \ref{sim_omega2}), we observe that all simulated trajectories converge to the equilibrium points and that there is not a heteroclinic orbit. When we fix $\omega$ and gradually increase $\alpha$, the behaviors of the trajectories progress from the dynamics that we observe in region $\Omega_1$ to those that we observe in $\Omega_2$ and finally to those in $\Omega_3$. Therefore, we conjecture that for each fixed $\omega$, there exists a value $\alpha_\omega$ of $\alpha$ such that the region $\Omega_2$ corresponds to the region with $2\omega \leq \alpha \leq \alpha_\omega$ and the region $\Omega_3$ corresponds to the region with $\alpha_\omega < \alpha$.

\begin{figure}[htp]
\centering
\includegraphics[width=\textwidth]{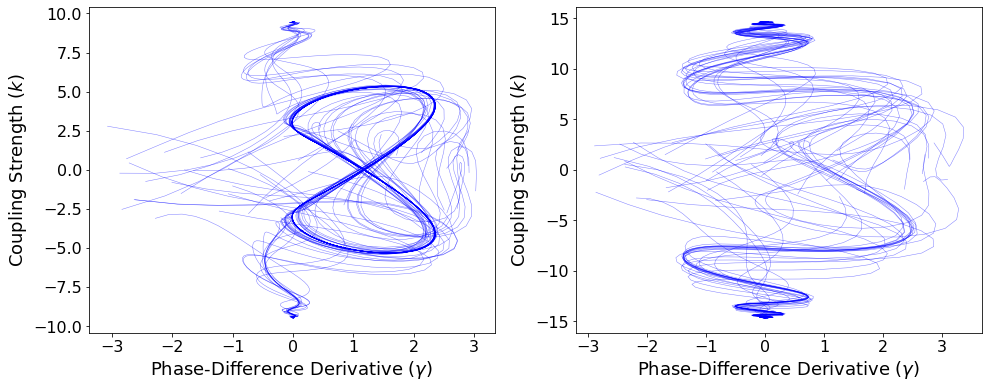}
\caption{(Left) Projection onto the $(\gamma,k)$ plane of simulated trajectories in the region $\Omega_2$ for the dynamical system (\ref{reduce1})--(\ref{reduce3}) of $N = 2$ coupled oscillators with parameters $m=1$, $\omega=3$, and $\alpha = 10$. (Right) Projection onto the $(\gamma,k)$ plane of simulated trajectories in the region $\Omega_3$ for the dynamical system (\ref{reduce1})--(\ref{reduce3}) with parameters $m=1$, $\omega=3$, and $\alpha = 15$.
}
\label{sim_omega2}
\end{figure}

We seek to approximate the three regions $\Omega_1$, $\Omega_2$, and $\Omega_3$ in the $(\alpha,\omega)$ plane to gain insight into $\alpha_\omega$. We do this with numerical simulations and use a coarse approach to check whether or not there exists a heteroclinic orbit that connects $P_2$ and $P_4$. We assume that $(\alpha,\omega) \in [0,36]\times[0,2\pi)$, and we then divide this rectangle into a grid with $150 \times 150$ points and consider the parameter values $(\alpha,\omega)$ at each grid point. For each value of ($\alpha, \omega$), because the phase $\phi$ has period $2\pi$, we consider the Poincar\'e section $\mathcal{P} = \{(\phi,\gamma,k)| \phi = 0 \}$. We pick 20 uniformly random initial values in the rectangle $R = \{0\}\times[-\pi,\pi)\times[-\pi,\pi)$ in $\mathcal{P}$. If the dynamical system (\ref{reduce1})--(\ref{reduce3}) has a heteroclinic orbit, then a small perturbation of the heteroclinic orbit will yield a trajectory that has multiple intersections with $\mathcal{P}$.

For each initial condition, we integrate for $1000$ time steps and we classify the resulting region based on the mean number of times that a trajectory intersects $\mathcal{P}$. In practice, we find that if a trajectory converges to an equilibrium point, then it only intersects $\mathcal{P}$ once or twice. Our approach only yields a rough estimate of the regions $\Omega_1$, $\Omega_2$, and $\Omega_3$ (see Figure \ref{parameter_region}); it does not precisely determine the boundaries between these regions.

\begin{figure}[htp]
\centering
\includegraphics[width=0.6\textwidth]{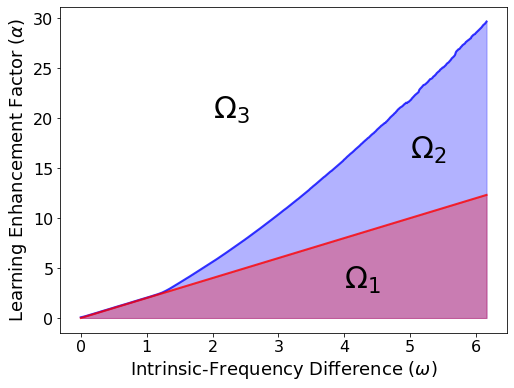}
\caption{Approximate regions in the $(\alpha,\omega)$ plane for which the dynamical system (\ref{reduce1})--(\ref{reduce3}) with $N = 2$ coupled oscillators has different dynamical properties. The region $\Omega_1$ consists of the parameter values $(\alpha,\omega)$ for which the dynamical system (\ref{reduce1})--(\ref{reduce3}) does not have any equilibrium points. The regions $\Omega_2$ and $\Omega_3$, respectively, consist of the parameter values $(\alpha,\omega)$ for which the dynamical system (\ref{reduce1})--(\ref{reduce3}) has equilibrium points with heteroclinic orbits and without heteroclinic orbits.}
\label{parameter_region}
\end{figure} 

%%%%

\section{A Preliminary Investigation of the Dynamical System (\ref{kuramoto_general_equation})--(\ref{hebb}) with Many Oscillators}\label{high_dim_section}

In Section \ref{low_dim_section}, we observed that the qualitative dynamics of the trajectories of the 3D dynamical system (\ref{reduce1})--(\ref{reduce3}) depend on the oscillators' intrinsic-frequency difference $\omega$ and the learning enhancement factor $\alpha$. We now use these insights to motivate our preliminary investigation of the general equations (\ref{kuramoto_general_equation})--(\ref{hebb}) for the Kuramoto model with inertia and Hebbian learning. In particular, we now consider high-dimensional situations.

We consider a specific setup for the intrinsic oscillator frequencies $\omega_i$ (with $i \in \{1, \ldots, N\}$) by sampling them randomly from a Gaussian distribution with $0$ mean and variance $\mathcal{N}(0, \sigma^2)$ for different values of $\sigma$. We expect that the variance $\sigma^2$ in the high-dimensional system (\ref{kuramoto_general_equation})--(\ref{hebb}) plays a role that is analogous to the intrinsic-frequency difference $\omega$ in the 3D system (\ref{reduce1})--(\ref{reduce3}). 

In our study of the low-dimensional  transverse system (\ref{reduce1})--(\ref{reduce3}), we let the oscillators have a homogeneous mass of $m = 1$. In the original high-dimensional system (\ref{kuramoto_general_equation})--(\ref{hebb}), we expect that the inertia terms play a significant role in the synchronization of the oscillators. In our numerical computations, we consider a coupled system of $N = 50$ oscillators. Olmi et al.~\cite{olmi2014} noted that by increasing the value of inertia, {one observes that the Kuramoto model with inertia (without any adaptation) has a partially synchronized state: in addition to the cluster of phase-locked oscillators with $\frac{d\phi_i}{dt} \approx 0$, there are also clusters of phase-locked oscillators with finite mean velocities $\frac{d\phi_i}{dt} \not \approx 0$. These additional clusters are called ``drifting coherent clusters'' of oscillators. Adaptive Kuramoto models without inertia also develop clusters of phase-locked oscillators \cite{berner_multiclusters, berner_adaptive}. Both with and without inertia, the formation of clusters of oscillators can depend on the adaptation rule. 

Niyogi and English \cite{niyogi_english_2009} observed for complete networks that the extension of the Kuramoto model with the adaptation rule in \eqref{hebb} induces two stable synchronized clusters in anti-phase when the learning rate is larger than a critical value. To account for the possibility of two phase-locked clusters of oscillators, we examine the order parameter
\begin{equation} \label{order}
    r_2(t) = \left|\frac{1}{N}\sum_{j=1}^{N}e^{2i \, \phi_j(t)} \right| \,.
\end{equation}
We perform numerical simulations to investigate how $r_2(t)$ changes with time. We consider two systems of $N = 50$ coupled oscillators. Suppose that the $i$th oscillator has mass $m_i = m$, so masses are homogeneous. We consider examples with light masses ($m = 1$) and heavy masses ($m = 100$). Each of the two systems is a dynamical system of $\frac{50(51)}{2} = 1275$ coupled differential equations. We set the initial phases of each oscillator to be evenly spaced in $[0, 2\pi)$, the initial phase derivatives to be $\phi_i'(0) = 0$, and the initial coupling strengths to be $K_{ij} = 1$. We investigate the effects of $\alpha$ and $\sigma^2$ on the order parameter $r_2(t)$ by fixing one of the two parameters and varying the other. We show our results in Figures \ref{high_dim_order_param} and \ref{high_dim_high_inertia_order_param}.

\begin{figure}[t]
\centering
\includegraphics[width = \textwidth]{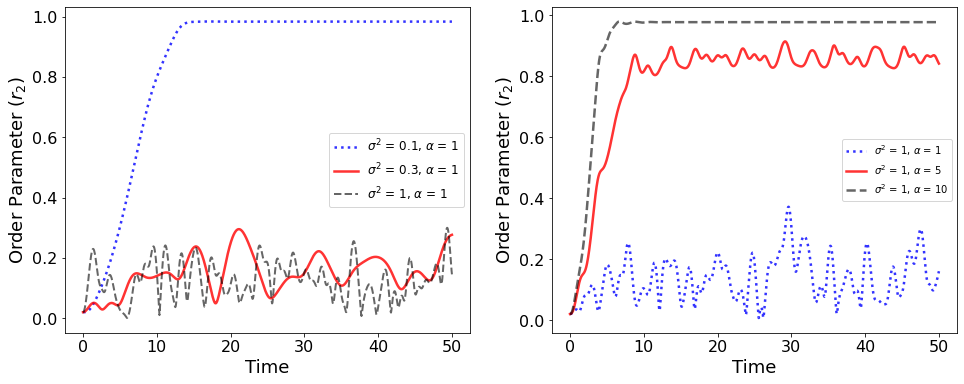}
\caption{The order parameter $r_2(t)$ for the system (\ref{kuramoto_general_equation})--(\ref{hebb}) with $N=50$ oscillators that each have a mass of $m =1$. 
(Left) We fix the learning enhancement factor to $\alpha = 1$ and consider variances of $\sigma^2 = 0.1$, $\sigma^2 = 0.3$, and $\sigma^2 = 1$. (Right) We fix the variance to be $\sigma^2 = 1$ and consider learning enhancement factors of $\alpha = 1$, $\alpha = 5$, and $\alpha = 10$. For each simulation, we draw a new set of natural oscillator frequencies from the specified distribution. Therefore, the order parameters for $\sigma^2 = 1$ and $\alpha = 1$ are different in the two panels.}  
\label{high_dim_order_param}
\end{figure}

\begin{figure}[htp]
\centering
\includegraphics[width = \textwidth]{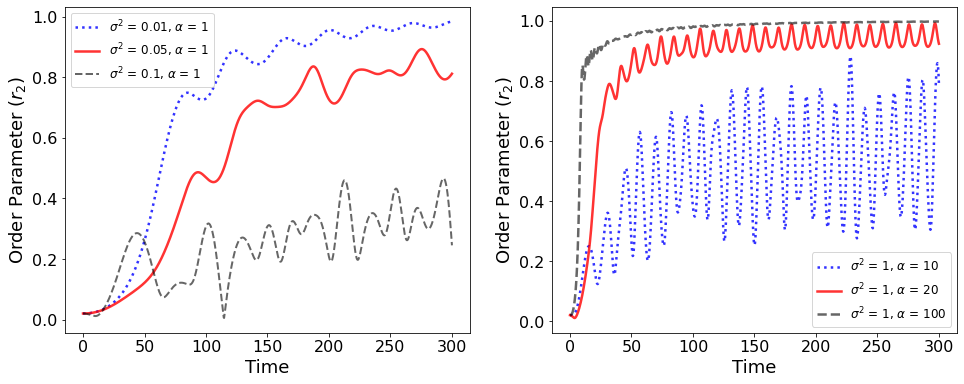}
\caption{The order parameter $r_2(t)$ for the system (\ref{kuramoto_general_equation})--(\ref{hebb}) with $N = 50$ oscillators that each have a mass of $m =100$. 
(Left) We the fix the learning enhancement factor to $\alpha = 1$ and consider variances of $\sigma^2 = 0.01$, $\sigma^2 = 0.05$, and $\sigma^2 =  0.1$. (Right) We fix the variance to be $\sigma^2 = 1$ and consider learning enhancement factors of $\alpha = 10$, $\alpha = 20$, and $\alpha = 100$.}  
\label{high_dim_high_inertia_order_param}
\end{figure} 

%%%%%

For both light masses and heavy masses, we observe that the order parameter $r_2(t) \rightarrow 1$ when $\sigma^2 \ll \alpha$, suggesting that almost all oscillators are either in a single fully synchronized cluster or that there are two phase-locked clusters. In both cases, we also observe that the oscillators are incoherent when $\sigma^2 \gg \alpha$. When the oscillators are not entirely in either one or two phase-locked clusters, we observe oscillations in the order parameter, with a more pronounced amplitude for heavy masses than for light masses. When $r_2(t)$ oscillates, we observe multiple drifting coherent clusters in addition to the two large phase-locked clusters. We observe more drifting coherent clusters for heavy masses than for light masses.
 
Heuristically, when adaptivity dominates inertia (specifically, when $m$ is small and $\alpha \gg \sigma^2$), the oscillators tends to form one or two phase-locked clusters because of Hebbian learning. When inertia is large, it delays the formation of such phase-locked clusters and yields small drifting coherent clusters that resemble that ones that Olmi et al. \cite{olmi2014} observed in a nonadaptive Kuramoto model with inertia. We thus observe oscillations in the order parameter $r_2(t)$.

The qualitative dynamics of the high-dimensional system (\ref{kuramoto_general_equation})--(\ref{hebb}), which we examined with 50 oscillators, are similar to those that we observed in the transverse two-oscillator system (\ref{reduce1})--(\ref{reduce3}). In the 50-oscillator system, for fixed values of $\sigma^2$ and $\omega$, as we increase the value of $\alpha$, the incoherent oscillators experience progressively more phase-locking until eventually most of the oscillators are in one or two phase-locked clusters. In the transverse two-oscillator system, for fixed $\sigma^2$ and $\omega$, as we increase the value of $\alpha$, progressively more trajectories converge to the equilibrium points $P_1$ and $P_3$.

%%%%%

\section{Conclusions}\label{conc}

We studied an adaptive Kuramoto model with inertia in which the coupling strengths between phase oscillators depend on a Hebbian learning rule. We mostly examined a system with $N = 2$ coupled oscillators. This yields a 5D dynamical system, which decouples into a 3D transverse system and a 2D longitudinal system. Our analysis and numerical simulations of the transverse system suggest that it has three different types of qualitative behavior, which depends on the learning enhancement factor $\alpha$ and the intrinsic-frequency difference $\omega$ between the two oscillators.

Our insights from the two-oscillator system suggest a choice of parameter values in high-dimensional systems. We conducted numerical simulations of a 50-oscillator system in which we drew the intrinsic frequencies of the oscillators from a Gaussian distribution with $0$ mean. {We observed that the variance of the oscillators' intrinsic frequencies in the high-dimensional system plays a role that is similar to the intrinsic-frequency difference $\omega$ in the low-dimensional system. As we increased the learning enhancement factor $\alpha$, we observed that the high-dimensional system of coupled oscillators transitions from an incoherent state into a partially phase-locked state with drifting coherent clusters and finally to a state with at most two almost fully phase-locked clusters. 

In neuroscience, long-term potentiation (LTP) synapses and long-term depression (LDP) synapses refer, respectively, to types of synapses in which presynaptic neurons repeatedly promote and inhibit postsynaptic neurons~\cite{bliss2011}. In a model of a neuronal system as a set of coupled oscillators, LTP describes a situation with all oscillators in phase and LDP describes a situation with oscillators split into two groups that are anti-phase with respect to each other~\cite{niyogi_english_2009}. Based on (1) our observation that the trajectories in our transverse two-oscillator system converge either to a periodic orbit or to one of two equilibrium points and (2) our definition of the order parameter $r_2(t)$, which captures the synchrony of two phase-locked clusters of oscillators, our Kuramoto model with inertia and Hebbian learning suggests that the behaviors of the oscillators are simplified analogues of the behaviors of LTP and LDP synapses in neuronal networks.

A natural extension of our work is the analysis of how changes in inertia affect the transition to phase-locked groups of oscillators for different values of the system parameters. To better understand the interaction between the large phase-locked clusters and the small drifting clusters, it is desirable to conduct a thorough investigation of the formation of drifting coherent clusters. In the transverse two-oscillator system (\ref{reduce1})--(\ref{reduce3}), it seems worthwhile to obtain an analytical approximation for how the boundary between the regions $\Omega_2$ and $\Omega_3$ changes with respect to changes in inertia. We hope that a better understanding of the demarcation between $\Omega_2$ and $\Omega_3$ can provide further insight into the qualitative dynamics in different regions of parameter space for the $N$-oscillator (i.e., high-dimensional) system (\ref{kuramoto_general_equation})--(\ref{hebb}).

%%%%%

\appendix

%%%%%

\section{Appendix A} \label{app}

In this appendix, we prove Proposition \ref{prop_eig}.

\begin{reprop}{1}
 Consider the dynamical system (\ref{reduce1})--(\ref{reduce3}) with $\alpha > 2\omega$. Let $u := \alpha + \sqrt{\alpha^2-4\omega^2}$ and $v := \alpha - \sqrt{\alpha^2-4\omega^2}$.
 The following statements hold:
\begin{enumerate}
    \item Let $\Gamma_1$ be the region in the $(u,v)$ plane with $0 \leq v \leq u \leq \frac{2(m^2-m+1)}{3m}$ that is bounded by the curves
    \begin{align*}
        v &= u-\frac{(m+1+\sqrt{4(m+1)^2-6m(u+2)})(\sqrt{4(m+1)^2-6m(u+2)}-2(m+1))^2}{54m^2} \,, \\
        v &= u-\frac{(m+1-\sqrt{4(m+1)^2-6m(u+2)})(\sqrt{4(m+1)^2-6m(u+2)}+2(m+1))^2}{54m^2} \,.
    \end{align*}
    If $(u,v) \notin \Gamma_1$, then the Jacobian matrix at the equilibria $P_1$ and $P_3$ has a negative real eigenvalue and two complex-conjugate eigenvalues with negative real part. Otherwise, the Jacobian matrix at the equilibria $P_1$ and $P_3$ has three negative real eigenvalues.
    \item Let $\Gamma_2$ be the region in the $(u,v)$ plane with $0 \leq v \leq u$ and $v \leq \frac{2(m^2-m+1)}{3m}$ that is bounded by the curves
    \begin{align*}
        u = v-\frac{(m+1+\sqrt{4(m+1)^2-6m(v+2)})(\sqrt{4(m+1)^2-6m(v+2)}-2(m+1))^2}{54m^2} \,, \\
        u = v-\frac{(m+1-\sqrt{4(m+1)^2-6m(v+2)})(\sqrt{4(m+1)^2-6m(v+2)}+2(m+1))^2}{54m^2} \,.
    \end{align*}
    If $(u,v) \notin \Gamma_2$, the Jacobian matrix at the equilibria $P_2$ and $P_4$ has a positive real eigenvalue and two complex-conjugate eigenvalues with negative real part. Otherwise, the Jacobian matrix at the equilibria $P_2$ and $P_4$ has one positive real eigenvalue and two negative real eigenvalues.
\end{enumerate}
\end{reprop}

\begin{proof}
We first consider the region $\Gamma_1$, which is the region in the $(\alpha,\omega)$ plane for which the equilibria $P_1$ and $P_3$ have three negative real eigenvalues. 
From equation (\ref{eig_P1P3}), we need to consider the values of $\alpha$ and $\omega$ for which the polynomial
\begin{align*}
      f(x) &=  2x(x+1)(mx+1)+(\alpha+\sqrt{\alpha^2-4\omega^2})(x+1) - (\alpha-\sqrt{\alpha^2-4\omega^2}) \notag \\
      &= 2mx^3+2(m+1)x^2+(u+2)x+u-v
\end{align*}
has three real roots. When $\alpha > 2\omega \geq 0$, it follows that $u \geq v \geq 0$ are real. Therefore, because $f(x)$ is a degree-3 polynomial with real coefficients, it must have either three real roots or one real root and two complex-conjugate roots. The boundary of the region $\Gamma_1$ occurs when $f(x)$ has a double root $\tilde{x}$. Therefore, we also consider
\begin{equation*}
    f'(x) = 6mx^2+4(m+1)x+u+2 \,.
\end{equation*}
The root $\tilde{x}$ must be a root of both $f(x) = 0$ and $f'(x) = 0$, so it must be a root of
\begin{equation*}
      Q(x) = 3f(x)-xf'(x) = 2(m+1)x^2+2(u+2)x+3(u-v) = 0\,.
\end{equation*}
We obtain $\tilde{x}$ by solving
\begin{align*}
    0 &= (m+1)f'(\tilde{x})-3m Q(\tilde{x}) \\ 
    &= (4(m+1)^2-6(u+2)m)\tilde{x}-(9m(u-v)-(u+2)(m+1)) 
\end{align*}
to yield
\begin{equation*}
    \tilde{x} = \frac{9m(u-v)-(u+2)(m+1)}{4(m+1)^2-6(u+2)m} \,.
\end{equation*}
The value $\tilde{x}$ must also be a root of $f'(x) = 0$, so
\begin{equation} \label{this}
    \tilde{x} = \frac{9m(u-v)-\left(\frac{4(m+1)^2-s}{6m}\right)(m+1)}{s} = \frac{-(m+1) \pm 2\sqrt{s}}{12m} \,,
\end{equation}
where $s := 4(m+1)^2-6(u+2)m$. Rearranging equation \eqref{this} yields
\begin{equation} \label{above}
    u-v = \frac{\pm s\sqrt{s}-3s(m+1)+4(m+1)^3}{54m^2} = \frac{(\pm\sqrt{s}+m+1)\left(\sqrt{s}\mp 2(m+1)\right)^2}{54m^2} \,.
\end{equation}
The choice of signs in \eqref{above} (where the upper and lower sign choices correspond) gives the two boundary curves in the proposition. For $\tilde{x}$ to be a real number, we require that $s \geq 0$, which implies that $u \geq \frac{2(m^2-m+1)}{3m}$. We obtain equation (\ref{eig_P2P4}) by swapping the variables $u$ and $v$ in equation (\ref{eig_P1P3}). We then obtain the boundary curves of the region $\Gamma_2$ using the same calculation with $u$ and $v$ swapped.

Observe that $f(0) = 2\sqrt{\alpha^2 - 4\omega^2}$ and $f(-1) = -(\alpha - \sqrt{\alpha^2 - 4\omega^2})$. Consequently, by the Intermediate Value Theorem, there exists at least one real root in the interval $[-1,0)$. By Vieta's Theorem, the sum of all of the roots is $-1 - \frac{1}{m} < -1$ and the product of the roots is $\frac{v-u}{2m} < 0$. Therefore, the sum of the other two roots must be negative and the product of the other two roots must be positive. This implies that the other two roots have negative real parts. Therefore, the eigenvalues of the Jacobian matrix at the equilibria $P_1$ and $P_3$ all have negative real parts. Similarly, let $g(x) := -2x(x+1)(mx+1) - (\alpha-\sqrt{\alpha^2-4\omega^2})(x+1)+(\alpha+\sqrt{\alpha^2-4\omega^2})$ be the left-hand side of equation (\ref{eig_P2P4}). Observe that $g(0) = 2\sqrt{\alpha^2-4\omega^2} >0$ and that $g(\frac{2\sqrt{\alpha^2-4\omega^2}}{\alpha-\sqrt{\alpha^2-4\omega^2}}) = -\frac{4\sqrt{\alpha^2-4\omega^2}(\alpha+\sqrt{\alpha^2-4\omega^2})}{(\alpha-\sqrt{\alpha^2-4\omega^2})^2} < 0$. Therefore, by the Intermediate Value Theorem, there exists at least one positive real root in the interval $(0, \frac{2\sqrt{\alpha^2-4\omega^2}}{\alpha-\sqrt{\alpha^2-4\omega^2}})$. By Vieta's Theorem, the sum of all of the roots is $-1 - \frac{1}{m} < 0$ and the product of the roots is $\frac{u-v}{2m} > 0$. Therefore, the sum of the other two roots must be negative and the product of the other two roots must be positive. This implies that the other two roots have negative real parts. 

\end{proof}

%%%%

\section*{Acknowledgements}

We thank Predrag Cvitanovi\'{c}, Christian Kuehn, and two anonymous referees for helpful comments.

%%%%%

%\bibliography{Project_Bank_Mason_Citation09.bib}

%%%%%


%%%%%

\end{document}